\documentclass[psamsfonts,11pt]{amsart}
\usepackage[margin=1in]{geometry} \usepackage{amsmath,amssymb,amsthm,
  amscd} \usepackage{amssymb,fge} \usepackage{tikz-cd}
\usepackage{mathtools}
\usepackage{comment}

\newtheorem{thm}{Theorem}
\newtheorem{cor}[thm]{Corollary}
\newtheorem{prop}[thm]{Proposition}
\newtheorem{lem}[thm]{Lemma}

\newtheorem{defn}[thm]{Definition}
\newtheorem{question}[thm]{Question}

\theoremstyle{definition}

\newtheorem{fact}[thm]{Fact}

\theoremstyle{remark}
\newtheorem{remark}[thm]{Remark}

\numberwithin{thm}{section}
\numberwithin{equation}{thm}

\newcommand{\bP}{\ensuremath{{\mathbb{P}}}}
\newcommand{\C}{\ensuremath{{\mathbb{C}}}}
\newcommand{\bZ}{\ensuremath{{\mathbb{Z}}}}

\renewcommand{\P}{\ensuremath{{\mathbb{P}}}}

\newcommand{\bQ}{\ensuremath{{\mathbb{Q}}}}
\newcommand{\bF}{\ensuremath{{\mathbb{F}}}}

\newcommand{\F}{\ensuremath{{\mathbb{F}}}}

\newcommand{\N}{\ensuremath{{\mathbb{N}}}}
\newcommand{\p}{\ensuremath{{\mathfrak{p}}}}

\newcommand{\fp}{\ensuremath{{\mathfrak{p}}}}
\newcommand{\fq}{\ensuremath{{\mathfrak{q}}}}

\newcommand{\fo}{\ensuremath{{\mathfrak{o}}}}

\newcommand{\cZ}{\mathcal Z}

\newcommand{\cK}{\mathcal K}
\newcommand{\cS}{\mathcal S}
\newcommand{\cU}{\mathcal U}

\newcommand{\cY}{\mathcal Y}
\newcommand{\cX}{\mathcal X}
\newcommand{\cW}{\mathcal W}

\newcommand{\cC}{\mathcal C}
\newcommand{\cP}{\mathcal P}
\newcommand{\lra}{\longrightarrow}
\newcommand{\Kbar}{\ensuremath {\overline K}}

\newcommand{\afp}[1]{| #1 |_{\fp}}

\renewcommand{\O}{\ensuremath{{\mathcal{O}}}}
\newcommand{\Fp}{{\mathbb F}_p}
\newcommand{\Fpbar}{\overline {{\mathbb F}}_p}

\DeclareMathOperator{\Supp}{Supp}

\DeclareMathOperator{\Gal}{Gal}

\DeclareMathOperator{\PGL}{PGL}

\DeclareMathOperator{\Aut}{Aut}

\begin{document}

\title{Integral Points in Orbits in Characteristic $p$}



\author{Alexander Carney}
\address{Alexander Carney\\Department of Mathematics\\ University of Rochester\\
Rochester, NY, 14620, USA}
\email{alexanderjcarney@rochester.edu}

\author{Wade Hindes}
\address{Wade Hindes\\Department of Mathematics\\ Texas State University\\
San Marcos, TX, 78666, USA}
\email{wmh33@txstate.edu}

\author{Thomas J. Tucker}
\address{Thomas J. Tucker\\Department of Mathematics\\ University of Rochester\\
Rochester, NY, 14620, USA}
\email{thomas.tucker@rochester.edu}



\date{}

\dedicatory{}

\begin{abstract}
  We prove a characteristic $p$ version of a theorem of Silverman on
  integral points in orbits over number fields and establish a
  primitive prime divisor theorem for polynomials in this setting.  In
  characteristic $p$, the Thue-Siegel-Dyson-Roth theorem is false, so
  the proof requires new techniques from those used by
  Silverman.  The problem is largely that isotriviality can arise in subtle ways, and we define and compare three different definitions of isotriviality for maps, sets, and curves. Using results of Favre and Rivera-Letelier on the
  structure of Julia sets, we
  prove that if $\varphi$ is a non-isotrivial rational function
  and $\beta$ is not exceptional for $\varphi$, then
  $\varphi^{-n}(\beta)$ is a non-isotrivial set for all sufficiently
  large $n$; we then apply diophantine results of Voloch and Wang that
  apply for all non-isotrivial sets.  When $\varphi$ is a polynomial,
  we use the non-isotriviality of $\varphi^{-n}(\beta)$ for large $n$
  along with a partial converse to a result of Grothendieck in
  descent theory to deduce the non-isotriviality of the curve
  $y^\ell = \varphi^n(x) - \beta$ for large $n$ and small primes
  $\ell \not= p$ whenever $\beta$ is not post-critical; this enables
  us to prove stronger results on Zsigmondy sets.  We provide some
  applications of these results, including a finite index theorem for
  arboreal representations coming from quadratic polynomials over
  function fields of odd characteristic.
\end{abstract}

\maketitle


\section{Introduction and Statement of Results}\label{intro}

In \cite[Theorem A]{SilInt}, Silverman proved the following theorem.

\begin{thm}  \label{SilThm} \cite[Theorem A]{SilInt}
  Let $\varphi \in \bQ(z)$ be rational function of degree at least 2,
  and let $\alpha \in \bP^1(\bQ)$.  If $\varphi^2 \notin \bQ[z]$, then
  the set 
  $\{ \varphi^n(\alpha) \; | \; n \in \bZ^+ \}$ contains only finitely many points in $\bZ$.  
\end{thm}

We prove that the analogous theorem holds for non-isotrivial rational
functions in $\Fp(t)$.  Recall that a rational function in
$\varphi \in \Fp(t)(z)$ is said to be isotrivial if there is a
$\sigma \in \overline{\Fp(t)}(z)$ of degree 1 such that
$\sigma \circ \varphi \circ \sigma^{-1} \in \Fpbar(z)$.  We prove the
following.  

\begin{thm} \label{main-simple}
  Let $\varphi \in \Fp(t)(z)$ be a non-isotrivial rational function of
  degree at least 2, and
  let $\alpha \in \bP^1(\Fp(t))$.  If $\varphi^2 \notin \Fp(t)[z]$, then
  $\{ \varphi^n(\alpha) \; | \; n \in \bZ^+ \}$ contains only finitely many points in $\Fp[t]$.  
\end{thm}


Silverman \cite{SilInt} also proves Theorem \ref{SilThm} over number
fields (see \cite[Theorem B]{SilInt}). Likewise, our most
general form of Theorem \ref{main-simple} is stated in terms of
$S$-integrality and isotriviality for rational functions defined over
finite extensions of $\F_p(t)$.  We will define $S$-integrality in the
next section (see Definition \ref{S-def}).  We give our more general
definition of isotriviality for rational functions here.

 \begin{defn}\label{iso1}
   Let $K$ be a finite extension of $\Fp(t)$ and let $\varphi$ be a
   rational function in $K(z)$.  We say that $\varphi$ is an
   isotrivial rational function
   if there exists $\sigma \in \overline{K}(z)$ of degree 1 such that
   $\sigma \circ \varphi \circ \sigma^{-1} \in \Fpbar(z)$.
 \end{defn}

 Also recall that for a rational function $\varphi \in K(z)$, a
 point $\beta \in \bP^1(K)$ is said to be {\bf exceptional} for $\varphi$ if its total orbit (both forward and backward) is finite. 
 However, for the maps that we consider, this amounts to 
 $\varphi^{-2}(\beta)= \{ \beta \}$ by Riemann-Hurwitz. 
 In particular, since totally inseparable maps are isotrivial (which may be seen by moving fixed points to $0$ and $\infty$), 
 we avoid the more exotic cases of exceptional points arising in positive characteristic; see, for instance, \cite{PolyIt}. With this in place, we state our general form of Theorem \ref{main-simple}.  

 \begin{thm}\label{main}
   Let $K$ be a finite extension of $\Fp(t)$, let $\varphi \in K(z)$
   be a non-isotrivial rational function with $\deg \varphi > 1$, let $S$ be a
finite set of places of $K$, and let $\alpha, \beta \in K$ where
$\beta$ is not exceptional for $\varphi$.   Then 
  $\{ \varphi^n(\alpha) \; | \; n \in \bZ^+ \}$ contains only finitely many points that are
  $S$-integral relative to $\beta$.  
  \end{thm}

  The main tools used in the proof of \cite[Theorem A]{SilInt} are
  from diophantine approximation.  Roughly, one takes an inverse image
  $\varphi^{-i} (\infty)$ that contains at least three points and
  applies Siegel's theorem on integral points for the projective line
  with at least three points deleted to conclude that that there only
  finitely many $n$ such that $\varphi^n$ are integral relative to
  $\varphi^{-i}(\infty)$ and thus only finitely many $n+i$ such that
  $\varphi^{n+i}(\alpha)$ is an integer.  Over function fields in
  characteristic $p$, the problem is more complicated since Roth's
  theorem is false; in fact, no improvement on Liouville's theorem is
  possible in general.  There is, however, a weaker version of
  Siegel's theorem, due to Wang \cite[Theorem in $\bP^1(K)$, Page
  337]{Wang1} and Voloch \cite{Voloch}, which states that, for
  function fields in characteristic $p$, there are finitely many
  $S$-integral points on the projective line with a non-isotrivial set
  of points deleted. (Note that this is strictly weaker than Siegel's
  theorem, since any set of three points is automatically isotrivial,
  and there are isotrivial sets of every countable cardinality.)  Basic functorial
  results on integral points thus imply that Theorem \ref{main} will
  hold whenever $\varphi^{-n}(\beta)$ is a non-isotrivial set.  In
  Theorem \ref{noniso-inverse}, we show that $\varphi^{-n}(\beta)$ is
  a non-isotrivial set for large $n$ whenever $\varphi$ is a
  non-isotrivial rational function and $\beta$ is not exceptional,
  using results of Favre and Rivera-Letelier \cite{FRL} on the structure of Julia
  sets at primes of genuinely bad reduction.

In the case where $\varphi$ is a polynomial of separable degree
greater than 1, we can prove a bit more than Theorem \ref{main}.  To
describe our result we need a bit of terminology.  For a sequence
$\{b_n\}_{n=1}^\infty$ of elements of a global field $K$, we say that a place
$\fp$ of $K$ is a {\bf primitive divisor} of $b_n$ if
 \[ \text {$v_\fp(b_n) > 0$ and $v_\fp(b_m) \leq 0$ for all $m <
     n$}. \] For a positive integer $\ell$, we say that $\fp$ is a
 {\bf primitive $\ell$-divisor} of $b_n$ if
 \[ \text {$\fp$ is a primitive divisor of $b_n$ and $\ell \nmid
     v_\fp(b_n)$}. \]
 Given a rational function $\varphi \in K(x)$ and points $\alpha,
 \beta \in K$, we obtain a sequence $\{\varphi^n(\alpha) -
 \beta\}_{n=1}^\infty$.  We define the Zsigmondy set $\cZ(\varphi,
 \alpha, \beta)$ (see \cite{Bang,Zsigmondy}) for $\varphi$,
 $\alpha$, and $\beta$ as 
 \[ \cZ(\varphi, \alpha, \beta) = \{ n \; | \; \text{$\varphi^n(\alpha) -
     \beta$ has no primitive divisors} \}. \]

 Likewise, for a positive integer $\ell$ and $\alpha$, $\beta$, and
 $\varphi$ as above, we define the
 $\ell$-Zsigmondy set $\cZ(\varphi, \alpha, \beta, \ell)$ for
 $\varphi$, $\alpha$, $\beta$, and $\ell$ as
 \[ \cZ(\varphi, \alpha, \beta, \ell) = \{ n \; | \; \text{$\varphi^n(\alpha) -
     \beta$ has no primitive $\ell$-divisors} \}. \]

 We will also need a precise definition of critical points to state
 our next theorem.  Let $\varphi$ be a rational function in $K(z)$.
 We let $\deg_s \varphi$ denote the degree of the maximal separable
 extension of $K(\varphi(z))$ in $K(z)$ and let
 $\deg_i \varphi = (\deg \varphi)/(\deg_s \varphi)$; note that
 $\deg_i \varphi$ is also the largest power $p^r$ of $p$ such that
 $\varphi$ can be written as $\varphi(z) = g(x^{p^r})$ for some
 rational function $g \in K(z)$.  For $\gamma \in \bP^1$, there
 are degree one rational functions $\sigma, \theta \in K(z)$ such that
 $\theta(0)=\gamma$ and $\sigma \circ \varphi \circ \theta (0) = 0$.  We may then write
 $\sigma \circ \varphi \circ \theta (z) = z^e g(z)$ for some rational function $g$
 such that $g(z) \not= 0$.  We call $e$ the {\bf ramification degree}
 of $\varphi$ at $\gamma$ denote it as
 $e_{\varphi}(\gamma/\varphi(\gamma))$.  We say that $\gamma$ is a
 {\bf critical point} of $\varphi$ if
 $e_{\varphi}(\gamma/\varphi(\gamma))  > \deg_i \varphi$.

 We let $O^+_\varphi(\alpha)$ denote the set
 $\{ \varphi^n(\alpha) \; | \; n \in \bZ^+ \}$, called the forward orbit of $\alpha$ with respect to $\phi$.  Moreover, we say that a point $\beta$
 is post-critical if there is a critical point $\gamma$ of $\varphi$
 such that $\beta \in O^+(\gamma)$.

 With this terminology, we have the following two theorems for polynomials. 

 \begin{thm}\label{Z1}
   Let $K$ be a finite extension of $\Fp(t)$, let $f  \in K(z)$ be a
   non-isotrivial polynomial with $\deg f  > 1$, and let $\alpha$ and
   $\beta$ be elements of $K$ such that $\alpha$ is not preperiodic,
   $\beta$ is not post-critical, and $\beta \notin O^+_f (\alpha)$.  Then for any prime
   $\ell \not= p$, the Zsigmondy set $\cZ(f , \alpha, \beta, \ell )$ is
   finite. 
   \end{thm}

   \begin{thm}\label{Z2}
     Let $K$ be a finite extension of $\Fp(t)$, let $f  \in
     K(z)$ be a non-isotrivial polynomial with $\deg f  >
     1$, and let $\alpha$ and $\beta$ be elements of $K$ such that
     $\alpha$ is not preperiodic, $\beta$ is not exceptional for $f $,
     and $\beta \notin O^+_f (\alpha)$.  Then the Zsigmondy set
     $\cZ(f , \alpha, \beta)$ is finite.
   \end{thm}

   Theorem \ref{main} is not true in general for isotrivial rational
   functions, and Theorems \ref{Z1} and \ref{Z2} are not true not in
   general for isotrivial polynomials (see \cite{Pezda}).  There are
   some results in the isotrivial case, however (see \cite{HSW}), and
   some of the techniques here do work for a wide class of isotrivial
   rational functions.  We may address these questions in a future paper.

   Theorem \ref{main} is proved by using two different notions of
   isotriviality.  The first is our Definition \ref{iso1} for functions. We now define an isotrivial set.
   Here we use a simple, if inelegant, definition rather than a
   slightly more technical one that generalizes to varieties other
   than $\bP^1$. Below we regard an element of $\Kbar(z)$ as a map
   from $\Kbar \cup \infty$ to itself.  

   \begin{defn}\label{iso2}
     Let $K$ be a finite extension of $\Fp(t)$ and let $\cS$ be a
     finite subset of $\Kbar \cup \infty$.  We say that $\cS$ is a
     isotrivial set if there exists $\sigma \in \Kbar(z)$ of degree 1 such
     that $\sigma (\cS) \subseteq \Fpbar \cup \infty$.  
   \end{defn}

   We note that if $\varphi$ is a non-isotrivial rational function the
   set $\varphi^{-1}(\beta)$ may still be an isotrivial set; for example any
   set of three or fewer elements is an isotrivial set, but there are
   non-isotrivial rational functions of degree 2 and 3.  
   
   Theorem \ref{Z1} is proved using a third notion of isotriviality,
   this time for curves.

   \begin{defn}\label{iso3}
     Let $K$ be a finite extension of $\Fp(t)$ and let $C$ be a curve
     defined over $K$.  We say that $C$ is an isotrivial curve if there is a
     curve $C'$ defined over a finite extension $k'$ of $K \cap
     \Fpbar$ and a  finite extension $K'$ of $K$ such that
   \[ C \times_K K' \cong C' \times_{k'} K'. \]
     \end{defn}
  
     An outline of the paper is as follows.  Throughout this paper,
     $K$ is a finite extension of $\Fp(t)$ as in Definitions
     \ref{iso1} ,\ref{iso2}, and \ref{iso3}.  In Section
     \ref{heights}, we introduce some basic facts about heights,
     integral points, and cross ratios that are used throughout the
     paper.  Following that, we prove Theorem \ref{noniso-inverse},
     which says that if $\varphi$ is a non-isotrivial rational
     function of degree greater than 1 and $\beta$ is not exceptional
     for $\varphi$, then $\varphi^{-n}(\beta)$ is a non-isotrivial set
     for all sufficiently large $n$.  The proof uses work of Baker
     \cite{Baker1} and Favre/Rivera-Letelier \cite{FRL} to produce
     elements in $\varphi^{-n}(\beta)$ whose $v$-adic cross ratio is
     not 1 at a place $v$ of bad reduction.  We then apply work of
     \cite{Wang1} (see also \cite{Voloch}) to give a quick proof of
     Theorem \ref{main} in Section \ref{proof-main}.  In Section
     \ref{curves}, we begin by proving Proposition \ref{non-one},
     which states that if the roots of a polynomial $F$ are are
     distinct and form a non-isotrivial set, then the curve $C$ given
     by $y^\ell = F(x)$ is a non-isotrivial curve when $\ell \not= p$
     is a prime that is small relative to the degree of $F$.  The
     techniques we use to do this build upon work in \cite{HJ}; 
     the idea is to use the Adjunction Formula to show that the projection map onto the $x$-coordinate 
     is the unique map $\theta: C\rightarrow \mathbb{P}^1$ of degree $\ell$
     up to change of coordinates on $\bP^1$ (see Lemma
     \ref{lem:mapunique}).  We then use Proposition \ref{non-one} and
     Theorem \ref{noniso-inverse} to show the non-isotriviality of
     curves associated to $\varphi^{-n}(\beta)$, where $\varphi$ is a
     non-isotrivial rational function of degree greater than 1 and
     $\beta$ is not exceptional for $\varphi$, in Theorem
     \ref{non-curve}.  In Section \ref{Z-proofs}, we prove Proposition
     \ref{most-gen}, which immediately implies Theorems \ref{Z1} and
     \ref{Z2}; the proof uses Theorem \ref{noniso-inverse} along with
     height bounds on non-isotrivial curves in characteristic $p$ due
     to Szpiro \cite{Szpiro} and Kim \cite{Kim1} (see Theorem
     \ref{SZ}).  Finally, in Section \ref{applications}, we present
     some applications of our results to other dynamical questions.

     We note that the proof of Theorem \ref{noniso-inverse} works the
     same for function fields in characteristic 0 as for function
     fields in characteristic $p$.  Theorems \ref{main}, \ref{Z1}, and
     \ref{Z2} all hold in stronger forms for function fields in
     characteristic 0, as proved in \cite{GNT}; the main difference
     here is that Yamanoi \cite{Yam} has proved the full Vojta
     conjecture for algebraic points on curves over function fields of
     characteristic 0 (see \cite{Vojta-ABC, Vojta-SV}), whereas
     Theorem \ref{SZ} is weaker than the full Vojta conjecture for
     algebraic points on curves over function fields of characteristic
     $p$.  Analogs of Theorems \ref{Z1} and \ref{Z2} have not yet been
     proved over number fields, except in some very special cases (see
     \cite{Bang,Zsigmondy, Schinzel, post, Rice}), but both theorems
     are implied by the $abc$ conjecture (see \cite{GNT}).

\vskip2mm
\noindent {\bf Acknowledgments.}  We would like to thank Rob
Benedetto, Dragos Ghioca, Minhyong Kim, Carlo Pagano, Joe
Silverman, Dinesh Thakur, Felipe Voloch, and Julie Wang for many helpful
conversations.  We give special thanks to Juan Rivera-Letelier, who
provided us with the argument for Proposition \ref{badcrossratio} and
without whose help this paper likely would not have been possible.  
     
\section{Preliminaries}\label{heights}
In this section we will review some terminology and results on
heights, integral points, and dynamics. For background on heights, see
~\cite{HindrySilverman,Lang, BG}. We set some notation below.

Throughout this paper, $K$ will denote a finite extension of $\Fp(t)$
and $k$ will denote the intersection $K \cap \Fpbar$.  Equivalently, $K$ is
the function field of a smooth, projective curve $B$ defined over
$k$.

\subsection{Places, heights, and reduction}
Let $M_K$ be the set of places of $K$, which corresponds to the set of closed points of $B$.

Since $K$ is a function field, we choose a place $\fq$ of $K$, denote 
\[
\fo_K=\{z\in K \; : \; v_\fp(z)\geq 0\text{ for all }\fp\neq\fq\},
\]
and let $k_\fp$ be the residue field $\fo_K/\fp$. 
Also, define the local degree of $\fp$ to be  
\[
  N_\fp =[k_\fp: k].
\]
Likewise, for each $\fp\in M_K$ we let $|\cdot|_{\fp}$ be a normalized absolute value such that the product formula
\[
\prod_{\fp\in M_K}|z|_{\fp}=1
\]
holds for all $z\in K$. Moreover, we define $K_{\fp}$ to be the completion of
$K$ with respect to $|\cdot|_{\fp}$ and define $\C_\fp$ to be the
completion of the algebraic closure of $K_{\fp}$. 

For $z\in K$, let $h(z)$ denote the logarithmic height of $K$. For $
\varphi \in K(z)$ with $\deg \varphi =d\geq 2$, 
let $h_\varphi(z)$ denote the Call-Silverman canonical height of $z$ relative to $\varphi$~\cite{CallSilverman}, defined by
\[
h_\varphi(z) = \lim_{n\to\infty}\frac{h(\varphi^n(z))}{d^n}.
\]
We will often write sums indexed by primes that 
satisfy some condition. These are taken to be primes of $\fo_K$. 
As an example of our indexing convention, observe that
\[
\sum_{v_\fp(z)>0} v_\fp(z)N_\fp\leq h(z) .
\]

We say that a rational function $\varphi \in K(z)$ has {\bf good
  reduction} at a place $\fp$ of $K$ if the map it induces on $\bP^1$
is non-constant and well-defined modulo $\fp$.  More precisely, we
write $\varphi(x) = f/g$, where all the coefficients of $f$ and $g$
are in $(\fo_K)_\fp$, and either $f$ or $g$ has at least one
coefficient in $(\fo_K)_\fp^*$.  We let $f_\fp$ and $g_\fp$ denote the
reductions of $f$ and $g$ at $\fp$.  We say that $\varphi$ has good
reduction at $\fp$ if $f_\fp$ and $g_\fp$ have no common root in the
algebraic closure of the residue field of $\fp$ and
$\deg (f_\fp / g_\fp) = \deg \varphi$.  We say that $\varphi$ has {\bf
  bad reduction} at $\fp$ if it does not have good reduction at $\fp$.
This notion is dependent on our choice coordinates.  We say that
$\varphi$ has {\bf potentially good reduction} at $\fp$ if there is a
finite extension $K'$ of $K$, a prime $\fq$ of $K'$ lying over $\fp$,
and a degree one rational function $\sigma \in K'(z)$ such that
$\sigma \circ \varphi \circ \sigma^{-1}$ has good reduction at $\fq$.
We say that $\varphi$ has {\bf genuinely bad reduction} at $\fp$ if
$\varphi$ does not have potentially good reduction at $\fp$.

\subsection{Integral points}
Let $S$ be a non-empty finite subset of $M_K$. The ring of $S$-integers in $K$ is defined to be 
\[
\fo_{K,S}:=\{z\in K:\afp{z}\le 1\text{ for all }\fp\notin S\}.
\]

Given a place $\fp$ of $K$ and two points $\alpha=[x_1:y_1]$ and $\beta=[x_2,y_2]$ in $\P^1(\C_{\fp})$, define the \emph{$\fp$-adic chordal metric} $\delta_{\fp}$ by
\[
\delta_{\fp}(\alpha,\beta)=\frac{|x_1y_2-y_1x_2|_{\fp}}{\max\{\afp{x_1},\afp{y_1}\}\cdot\max\{\afp{x_2},\afp{y_2}\}}.
\]
Note that we always have $0\le \delta_{\fp}(\alpha,\beta)\le1$, and that $\delta_{\fp}(\alpha,\beta)=0$ if and only if $\alpha=\beta$. Then the ring $\fo_{K,S}$ is equivalent to the set which is maximally distant from $\infty$ outside of $S$, i.e. the set of $z\in K$ such that 
\[
\delta_{\fp}\left(z,\infty\right)=\delta_{\fp}\left([z:1],[1,0]\right)=1
\]
for all $\fp\notin S$.

We can now extend our definition of $S$-integrality to to any divisor
$D$ on $\bP^1$ that is defined over $K$.

\begin{defn}\label{S-def}
  Fix a non-empty finite set of places $S\subset M_K$.  Let $D$ be an
  effective divisor on $\bP^1$ that is defined over $K$.  Then
  $\alpha \in \bP^1(K)$ is $S$-integral relative to $D$ provided that for
  all places $\fp\notin S$, all $\tau\in\Gal(\Kbar /K)$, and all
  $\beta \in \Supp D$, we have
\[
\delta_{\fp}\left(\alpha,\tau(\beta)\right)=1.
\]
\end{defn}

For affine coordinates $[\alpha:1] \in \bP^1(K)$ and a divisor $D$ defined over
$K$ that does not contain the point at infinity in its support, the
statement that $[\alpha:1]$ is $S$-integral relative to $D$ is equivalent to
\begin{eqnarray*}
\afp{\alpha -\tau(\beta)}\ge 1&\text{  if  }&\afp{\tau(\beta)}\le1,\text{ and}\\
\afp{\alpha} \le 1 & \text{  if  }&\afp{\tau(\beta)}>1
\end{eqnarray*}
for all $\fp\notin S$, all $\tau\in\Gal(\Kbar/K)$, and all $[1:\beta]
  \in \Supp D$.

Let $\theta$ be a linear fractional change of coordinate on $\P^1(\bar
K)$. Then $\alpha$ is $S$-integral relative to $\beta$ if and only if
$\theta(\alpha)$ is $S$-integral relative to $\theta(\beta)$ provided we allow
an enlargement of $S$ depending only on $\theta$. We prove a variant
of this statement for any $\theta\in K[x]$ later in the paper.  The
following is a simple and standard consequence of our definition of
$S$-integrality (see \cite[Corollary 2.4]{Sookdeo}, for example).
Recall that for a point $\alpha \in \bP^1(K)$, the divisor
$\varphi^*(\alpha)$ is defined as $\sum_{\varphi(\beta) = \alpha}
e_\varphi(\beta/\alpha) \beta$.  

\begin{lem}\label{funct}
  Let $\varphi \in K(x)$ and $S$ be a set of primes containing all the
  primes of bad reduction for $\varphi$.  Then, for any
  $\alpha, \gamma \in \bP^1(K)$, we have that $\varphi(\gamma)$ is
  $S$-integral relative to $\alpha$ if and only if $\gamma$ is $S$-integral
  relative to $\varphi^* (\alpha)$. 
\end{lem}

\subsection{The cross ratio}

Let $| \cdot |$ be a non-Archimedean absolute value on a field $L$.  For any distinct $x_1, x_2, y_1, y_2 \in L$ we define:
\[ (x_1, x_2; y_1, y_2) = \frac{|x_1 - y_2| |x_2 - y_1|}{|x_1 - y_1|
    |x_2 - y_2|}. \]

We may extend this to points in $x_1, x_2, y_1, y_2 \in L \cup \infty$
by eliminating the terms involving $\infty$; for example,
\[ (\infty, x_2; y_1, y_2) = \frac{|x_2 - y_1|}{ |x_2 - y_2|}. \]

Importantly, for $\sigma \in\PGL_2(L)$, we have $[z_1,z_2;z_3,z_4] =
[\sigma z_1,\sigma z_2;\sigma z_3,\sigma z_4]$.   This is easily seen
by noting that an element of $\PGL_2(L)$ is a composition of
translations, scaling maps, and the map sending every element to its
multiplicative inverse, and that $[z_1,z_2;z_3,z_4]$ is invariant
under all these types of maps.

We will use the following two lemmas for points $x_1, x_2, y_1, y_2 \in
L$.  The first lemma is immediate.
\begin{lem}\label{con}
  Suppose that $|x_1| < |y_1| < |x_2| < |y_2|$.  Then
  \[  (x_1, x_2; y_1, y_2) = \frac{|y_2| |x_2|}{|y_1| |y_2|} > 1. \]
 \end{lem}

 \begin{lem}\label{two}
   Suppose that there are points $a_1, a_2 \in L$ such that $|x_1 - a_1|,
   |y_1 - a_1| <
   |a_1 - a_2|$ and $|x_2 - a_2|, |y_2 - a_2| < |a_1 - a_2|$.   Then
   \[  (x_1, x_2; y_1, y_2) > 1 .\]
 \end{lem}
 \begin{proof}
   After a translation, we may assume that $a_1 = 0$.  Then $|x_1|,
   |y_1| < |a_2|$ and $|x_2|, |y_2| = |a_2|$.  Thus, we have
  \[  (x_1, x_2; y_1, y_2) = \frac{|a_2| |a_2|}{|x_1 - y_1| |x_2 - y_2|} > 1. \]
\end{proof}

\begin{remark}
  The cross ratio of $x_1, x_2, y_1, y_2$ is often defined without
  taking absolute values, i.e. as 
  \[ \frac{(x_1 - y_2) (x_2 - y_1)}{(x_1 - y_1)
      (x_2 - y_2)}.\]
  The advantage of the definition we use is that it extends to points
  in Berkovich space (see \cite{FRL}).  While we do not use this
  extension, it can be used to give a quick proof of our Proposition
  \ref{badcrossratio}.  We give a slightly longer proof that we think
  may be more accessible for some readers.  
  \end{remark}


\section{Non-Isotriviality of inverse images}\label{noniso-sec}
In this section, we will prove the following theorem.

\begin{thm}\label{noniso-inverse}
Let $\varphi \in K(z)$ have $\deg \varphi > 1$.  Suppose that
$\varphi$ is not isotrivial and that $\beta$ is not exceptional for
$\varphi$.  Then for all sufficiently large $n$ the set
$\varphi^{-n}(\beta)$ is not an isotrivial set.   
\end{thm}

We will derive Theorem \ref{noniso-inverse} from the following
proposition.  

\begin{prop}\label{badcrossratio}
  Suppose $\varphi\in K(z)$ has genuinely bad reduction at the prime
  $\fp$.  Let $| \cdot |$ be an extension of $|\cdot|_{\fp}$ to 
  $\C_\fp$.  Then for any non-exceptional $\alpha\in K$, and for all
  sufficiently large $n$, there are elements
  $z_1,z_2,z_3,z_4\in\varphi^{-n}(\alpha)$ such that
$$(z_1,z_2;z_3,z_4) > 1.$$
\end{prop}

\begin{proof}
We work over the non-Archimedean complete field $\C_\fp$, and consider
the dynamical system induced by $\varphi$ on the Berkovich projective
line $\P^{1,an}$.  We will use some basic facts about the topology of
the Berkovich projective line, including the classification of points
as Type I, II, III, or IV; see \cite{BR-Book} or \cite{BenedettoBook} for a detailed
description of the topology of the Berkovich projective line.

By \cite[Th\'eor\`eme E]{FRL} (see also \cite[Theorem 8.15]{BenedettoBook}), bad reduction implies that the 
equilibrium measure $\rho_\varphi$ is non-atomic. 
Thus, there are four or more points all of the same type (I, II, III, or IV) in the support of $\rho_\varphi$.

Since $\rho_\varphi$ is non-atomic and the inverse images of a
non-exceptional point equidistribute we have the following fact.

\begin{fact}\label{fact}
For any $\gamma$ in the support
of $\rho_\varphi$, any open subset $U$ containing $\gamma$, an any positive integer
$m$, there is an $N$ such that $U \cap \varphi^{-n}(\beta)$ contains
$m$ or more points for all $n \geq N$.  
\end{fact}

We also have the following basic facts about the topology of
$\P^{1,an}$.

\begin{fact}\label{II}
  Let $\xi(a,r)$, where $a \in K$ and $r > 0$, be a point of Type II
  or Type III corresponding to the disc $\{x\in K \; | \; |x-a| \leq r \}$.
  Then for any $\epsilon  > 0$, there is an open set $U\subset \P^{1,an}$ with $\xi(a,r) \in U$
  such that every point $x$ of Type I in $U$ satisfies $r  - \epsilon
  < |x-a| < r + \epsilon$.  
 \end{fact}

\begin{fact}\label{I}
Let $a_1$ and $a_2$ be any two points of the same type in $\P^{1,an}$, 
which are not concentric Type II or III points. Then there exist open
sets $U_1$ and $U_2$ with $a_1 \in U_1$ and $a_2 \in U_2$ such that
$U_1 \cap \bP^{1}(\C_\fp)$ and $Y_2 \cap \bP^{1}(\C_\fp)$ are disjoint open discs.
\end{fact}

\begin{proof}
  Since $a_1$ and $a_2$ are not concentric, $a_1\wedge a_2$, the
  unique point
  $[a_1, \infty] \cap [a_2,\infty] = [a_1 \wedge a_2, \infty]$, is not
  equal to $a_1$ or $a_2$ (see \cite{FRL}). Now let $D_i$ be the open
  disc corresponding to any Type II point in the open interval
  $(a_i,a_1\wedge a_2)$, for $i=1,2$.  Then there are open sets $U_i$
  such that $U_i \cap \bP^{1}(\C_\fp) = D_i$.
\end{proof}



Suppose that $\rho_\varphi$ contains two non-concentric points $z_1, z_2$ 
of the same type.  Then, by Facts \ref{fact} and \ref{I}, for all sufficiently large $n$
there must be open discs $D(a_1, r_1)$ and $D(a_2, r_2)$ with
$|a_1 - a_2| > \max \{r_1, r_2\}$ and points
$x_1, x_2, y_1, y_2 \in \varphi^{-n}(\beta)$ with
$x_1, y_1 \in D(a_1, r_1)$ and $x_2, y_2 \in D(a_2, r_2)$.  By Fact
\ref{two}, we have
\[ (x_1, x_2; y_1, y_2) > 1, \]
proving the proposition.

Now suppose that $\rho_\varphi$ contains four concentric points of Type II or
Type III, corresponding to
closed discs ${\overline D}(a, r_i)$, for $i= 1, 2, 3, 4$, for some
fixed $a$.  We suppose that $r_1 < r_2 < r_3 < r_4$, and after an
affine change of coordinates, we may suppose that $a = 0$.  By Facts
\ref{fact} and \ref{II}, for any $\epsilon > 0$, there must be an $n$ such that
$\varphi^{-n}(\beta)$ contains points $z_1, z_2, z_3, z_4$ with
$|z_i|$ within $\epsilon$ of $r_i$ for each $i$.  Choosing $\epsilon$
appropriately, we will then have $|z_1| < |z_2| < |z_3| < |z_4|$.
Then $(z_1, z_3; z_2, z_4) > 1$ by Lemma \ref{con}.

\end{proof}


\begin{proof}[Proof of Theorem \ref{noniso-inverse}]
By \cite[Theorem 1.9]{Baker1}, since $\varphi$ is non-isotrivial, it
must have genuine bad reduction over some prime $\fp$. Then we may
apply Proposition \ref{badcrossratio} to obtain four points in $\varphi^{-n}(\beta)$
with cross ratio greater than one for any sufficiently large
$n$. Since the cross ratio of four points in $\Fpbar \cup \infty$ is
always 1 and  the cross ratio is invariant under change of
coordinate, we see then that $\varphi^{-n}(\beta)$ is a non-isotrivial
set for all sufficiently large $n$.  
\end{proof}

\section{Proof of Theorem \ref{main}}\label{proof-main}
We will use the following  theorem due to Wang \cite[Theorem in
$\bP^1(K)$, Page 337]{Wang1} and Voloch \cite{Voloch}.  

\begin{thm}\label{W} Let $D$ be an effective divisor on $\bP^1$ that is defined
  over $K$.  If the points in $Supp\, D$ form a non-isotrivial set, then
  the set of points in $\bP^1(K)$ that are $S$-integral relative to
  $D$ is finite.
\end{thm}

The corollary below follows easily.  
\begin{cor}\label{easy-cor}
  Let $\varphi \in K(z)$, let $\beta \in K$.  Suppose that there is
  some $i$ such that $\varphi^{-i}(\beta)$ is not an isotrivial set.
  Then for any $\alpha \in K$, the forward orbit $O^+_\varphi(\alpha)$
  contains only finitely many points that are $S$-integral relative to
  $\beta$.
  \end{cor}
  \begin{proof}
    We may extend $S$ to contain all the primes of bad reduction for $\varphi$.
  The set of iterates $\varphi^{n-i}(\alpha)$ that are
  $S$-integral relative to $(\varphi^i)^*(\beta)$ is finite by Theorem
  \ref{W}, so by Lemma \ref{funct}, the set of points
  $\varphi^n(\alpha)$ that are $S$-integral relative to $\beta$ must
  be finite.   
\end{proof}

The proof of Theorem \ref{main} is now easy.  

\begin{proof}[Proof of Theorem \ref{main}]
  By Theorem \ref{noniso-inverse}, there is some $i$ such that
  $\varphi^{-i}(\beta)$ is not an isotrivial set.  Applying Corollary
  \ref{easy-cor} then gives the desired conclusion.  
  \end{proof}





  \section{Non-isotriviality of certain curves}\label{curves}

  Let $\pi: C \lra \bP^1$ be a separable nonconstant morphism defined
  over $K$.  We define the {\bf ramification locus} of $\pi$ to be the
  support of $\pi(R_\pi)$, where $R_\pi$ is the ramification divisor
  of $\pi$.  If the ramification locus of $\pi$ is an isotrivial set,
  then it follows from descent theory (see \cite{SGA1}, for example)
  that $C$ must be isotrivial.  On the other hand, given any finite
  subset $\cU$ of $\bP^1$, one can use interpolation to construct a
  nonconstant separable morphism $f: \bP^1 \lra \bP^1$ such that that
  the ramification locus of $f$ contains $\cU$; thus, there are
  isotrivial curves that admit nonconstant separable morphisms
  $\pi: C \lra \bP^1$ such that the ramification locus of $\pi$ is a
  non-isotrivial set.  We can show, however, that if the degree of
  $\pi: C \lra \bP^1$ is a prime $\ell \not= p$ that is small relative
  to the genus of $C$ and the ramification locus of $\pi$ is a
  non-isotrivial set, then $C$ must indeed be a non-isotrivial curve.
  This enables us to prove Theorem \ref{non-curve}, which gives rise
  to diophantine estimates used in the proofs of Theorems \ref{Z1} and
  \ref{Z2}.  
  The technique here is similar to that of \cite{HJ}. We begin with a
  lemma about uniqueness of low prime degree maps on curves of high genus.

\begin{lem}{\label{lem:mapunique}}
Let $C$ be a curve of genus $g$ over $K$ and let $\ell$ be a prime
such that $(\ell - 1)^2 < g$ and $\ell \not= p$.  Suppose there is morphism $\theta_1: C \lra
\bP^1$ of degree $\ell$.  Then for any morphism $\theta_2: C \lra \bP^1$
of degree $\ell$, there is an automorphism
$\lambda: \bP^1 \lra \bP^1$ such that $\theta_2 = \lambda \circ \theta_1$.   
\end{lem}
  \begin{proof} Suppose that $g>(\ell-1)^2$ and that $\theta_2:C\rightarrow\mathbb{P}^1$ is  another map of degree $\ell$ on $C$. Then we have a map $(\theta_1,\theta_2):C\rightarrow\mathbb{P}^1\times\mathbb{P}^1$; let $\widetilde{C}$ be the image of this map. If $(\theta_1,\theta_2)$ is injective, then $\widetilde{C}$ also has genus $g$; see \cite[Theorem II.8.19]{Hartshorne}. On the other hand, $\widetilde{C}$ is a curve of bidegree $(d_1,d_2)$ in $\mathbb{P}^1\times\mathbb{P}^1$ for some $d_i\leq\ell$. Hence, the Adjunction Formula implies that $g=(d_1-1)(d_2-1)\leq(\ell-1)^2$, a contradiction; see \cite[Example V.1.5.2]{Hartshorne}. Therefore, $(\theta_1,\theta_2)$ is not an injection. However, we have a commutative diagram 
\begin{equation*}
\begin{tikzcd}[column sep=3em, row sep=3em]
&  \arrow[dl,swap,"\theta_1"] C \arrow[d,"(\theta_1\text{,}\theta_2)" description] \arrow[dr,"\theta_2"] & \\
  \mathbb{P}^1& \arrow[l,swap,"\pi_1"]\widetilde{C}\arrow[r,"\pi_2"] & \mathbb{P}^1
\end{tikzcd} 
\end{equation*}
where the $\pi_i$ are the restrictions of the natural projections $\pi_i:\mathbb{P}^1\times\mathbb{P}^1\rightarrow\mathbb{P}^1$ to $\widetilde{C}$. Therefore, \[\deg(\pi_1)\cdot\deg((\theta_1,\theta_2))=\deg(\theta_1)=\ell=\deg(\theta_2)=\deg(\pi_2)\cdot\deg((\theta_1,\theta_2)).\] 
However, $(\theta_1,\theta_2)$ is not injective, so that $\deg((\theta_1,\theta_2))>1$. Therefore, $\deg((\theta_1,\theta_2))=\ell$, since $\ell$ is prime. Hence, $\deg(\pi_1)=1=\deg(\pi_2)$, and both $\pi_i$ are isomorphisms \cite[Corollary 2.4.1]{SilvEll}. In particular, $\pi_2\circ\pi_1^{-1}=\lambda$ is a linear fractional transformation, and $\theta_2=\lambda\circ \theta_1$ as claimed.     
\end{proof}  

\begin{thm}\label{converse}
  Let $C$ be a curve of genus $g$ over $K$ and let $\ell$ be a prime
  such that $(\ell - 1)^2 < g$ and $\ell \not= p$.  Suppose there is
  morphism $\theta: C \lra \bP^1$ of degree $\ell$ such that the
  ramification locus of $\theta$ is a non-isotrivial set.  Then $C$
  is a non-isotrivial curve.
\end{thm}
\begin{proof}

  Suppose that $C$ is isotrivial; then there are finite extensions
  $K'$ of $K$ and $k'$ of $k$ such that there is a model $\cC$ for
  $C \times_K K'$ over the $k$'-curve $X$ corresponding to the
  function field $K'$ such that for any place
  $t \in X(\overline{k'})$, the curve $\cC_t \times_{k(t)} L$ is
  isomorphic to $C \times_K L$, where $k(t)$ is the field of
  definition of $t$ and $L = K' \cdot k(t)$.  Let $\cP$ be a model for
  $\bP^1$ over $X$.  Then, for all but finitely many places
  $t \in X(\overline{k'})$, the morphism $\theta$ specializes to a
  degree $\ell$ morphism $\theta_t: \cC_t \lra \bP^1_{k(t)}$ defined over $k(t)$.  Let
  $\theta_2 = \theta_t \times_{k(t)} L$.  Since $\theta_2: C \lra \bP^1$
  has degree $\ell$, and $(\ell - 1)^2 < g$, there is a
  $\lambda\in\PGL_2(\overline{K})$ such that
  $\theta_2 = \lambda \circ \theta$, by Lemma \ref{lem:mapunique}.  But
  $\lambda$ must take the ramification locus of $\theta$ to the
  ramification locus of $\theta_2$, which is defined over $k'$.
  Hence, the ramification locus of $\theta$ must be isotrivial.  That
  gives a contradiction.
\end{proof}

\begin{cor}\label{non-one}
  Let $F$ be a polynomial over $K$ without repeated roots such that the roots
  of $F$ form a non-isotrivial set.  Let $\ell$ be a prime number such
  that $\ell \not= p$ and $\ell - 1 < \deg F/2  - 1$.  
  Then the curve $C$ given by $y^\ell = F(x)$ is not isotrivial.  
\end{cor}
\begin{proof}
  Let $\theta: C \lra \bP^1$ be the map coming from projection onto
  the $x$-coordinate.  Then $\deg \theta = \ell$.  Since the genus
  of $C$ is at least $(\ell -1) \deg F/2 - (\ell -1)$ by
  Riemann-Hurwitz  and the ramification locus of $\theta$ includes the
  roots of $F$ (note: it will be larger than that if $\theta$
  also ramifies over the point at infinity), applying Theorem
  \ref{converse} shows that $C$ is not isotrivial.  
\end{proof}

  As mentioned above, there are obvious examples of maps $\pi: C \lra \bP^1$, where
  $C$ is isotrivial but the ramification locus of $\pi$ is not, but we
  have not found examples of isotrivial curves of the specific form
  $y^m = F(x)$, for $F$ a polynomial with distinct roots that form a non-isotrivial
  set and $m$ is an integer greater than 1 that is not a power of $p$.  

  \begin{question}\label{m}
    Does there exist an isotrivial curve of the form $y^m = F(x)$,
    where $F$ is a polynomial with distinct roots that form a
    non-isotrivial set and $m$ is an integer greater than 1 that is
    not a power of $p$?
  \end{question}

Corollary \ref{non-one} and the techniques of \cite{London} can be
used to show that when $p$ is odd and $m$ is even, the answer to
Question \ref{m} is ``no''; we cannot however rule out examples where
$m$ is odd or $p=2$.

We are now ready to prove a theorem guaranteeing the non-isotriviality
of certain curves obtained by taking inverse images of points under
iterates of a non-isotrivial rational function.  

\begin{thm}\label{non-curve}
  Let $\varphi \in K(x)$ be a non-isotrivial rational function.  Let $\beta \in K$ be
  non-exceptional for $\varphi$.  Then for any $\ell \not= p$, there
  is an $n$ such that the curve given by
\[ y^\ell = \prod_{\substack{ \gamma \in \Kbar \\ \varphi^n(\gamma) =
      \beta}}(x- \gamma) \]
 (where the product $\prod_{\substack{ \gamma \in \Kbar \\ \varphi^n(\gamma) =
      \beta}}(x- \gamma)$ is taken without multiplicities) is not an isotrivial curve.
 \end{thm}
\begin{proof}
  If $\infty \notin \varphi^{-n}(\beta)$ for any $n$, then this is
  immediate from Corollary \ref{non-one} and Theorem
  \ref{noniso-inverse}.  Otherwise, since $\deg_s \varphi > 1$
  (because purely inseparable rational functions are isotrivial) and
  $\beta$ is not exceptional for $\varphi$, there is some $m$ such
  that $\varphi^{-m}(\beta)$ contains at least three points.  Thus,
  there is some point $\beta' \in \varphi^{-m}(\beta)$ such that
  $\infty \notin \varphi^{-n}(\beta)$ for any $n$.  Then there is some
  $m'$ such that $\varphi^{-m'}(\beta')$ is not isotrivial by Theorem
  \ref{noniso-inverse}, and since the set of points other than
  $\infty$ in $\varphi^{-(m+m')}(\beta)$ contains
  $\varphi^{-m'}(\beta')$, this set is non-isotrivial as well, so the
  curve given by
\[ y^\ell = \prod_{\substack{ \gamma \in \Kbar \\ \varphi^{m+m'}(\gamma) =
      \beta}}(x- \gamma) \]
is not an isotrivial curve by Corollary \ref{non-one}.
\end{proof}

The second author conjectured \cite[Conjecture 3.1]{London} that when
$\varphi$ is a non-isotrivial polynomial of degree prime to $p$ and $\beta$
is not postcritical for $\varphi$, then for some $n$ and some $\ell$ prime to
$p$, the curve
\[ y^\ell = \prod_{\substack{ \gamma \in \Kbar \\ \varphi^n(\gamma) =
      \beta}}(x- \gamma) \]
is not isotrivial.   Theorem \ref{non-curve}  answers this with many of the hypotheses
removed.  Note that by taking the product without multiplicities, we
essentially remove the issue of $\beta$ being postcritical.  We note
that Ferraguti and Pagano have proved Theorem \ref{non-curve} in the
special case where $\varphi$ is a quadratic polynomial, $\ell =2$, and
$p \not= 2$ (see \cite[Theorem 2.4]{FP}).

\section{Proof of Theorems \ref{Z1} and \ref{Z2}} \label{Z-proofs}

Theorems \ref{Z1} and \ref{Z2} will both follow from the following more general
statement.

\begin{prop}\label{most-gen}
  Let $f \in K[x]$ be non-isotrivial with $\deg f > 1$ and let $\ell
  \not= p$ be
  a prime number.  Let $\alpha, \beta \in K$ where $\beta \notin O^+_\varphi(\alpha)$
  and $\alpha$ is not preperiodic.  Suppose that for some $r$, there
  is a $\gamma \in f^{-r}(\beta)$ such that $\gamma$ is not
  postcritical and $e_{f^r}(\gamma/\beta)$ is prime to $\ell$.  Then
  $\cZ(f, \alpha, \beta, \ell)$ is finite.
 \end{prop}

We will prove Proposition \ref{most-gen} by combining effective forms
of the Mordell Conjecture over function fields  (see \ref{SZ})
with Theorem \ref{non-curve} and the following 
lemma from \cite[Lemma 5.2]{BT} (see also \cite[Proposition 5.1]{GNT}). Note that while this lemma is stated in characteristic 0 
in \cite{BT}, the proof is the same word-for-word for finite
extensions of $\Fp(t)$.  

\begin{lem}\label{old}
Let $f\in K[x]$ with $d=\deg(f)\geq 2$. Let $\alpha\in K$ with $h_f(\alpha)>0$.  Let
  $\gamma_1, \gamma_2 \in K$ such that $\gamma_2\notin\O_f(\gamma_1)$. 
  For $n>0$, let $\cX(n)$ denote the
  set of primes $\p$ of $\fo_K$ such that
\[ 
\min(v_\p(f^m(\alpha)-\gamma_1), v_\fp(f^n(\alpha) - \gamma_2)) > 0
\]
for some $0 < m < n$. Then for any $\epsilon >0$, we have
\begin{equation*}
\sum_{\p\in\cX(n)} N_\p\leq \epsilon d^n h_f(\alpha)+ O_\epsilon(1).  
\end{equation*}
for all $n$.
\end{lem}

The next result we use follows from (any of the) effective forms of the Mordell Conjecture over functions fields \cite{Kim1,Moriwaki,Szpiro}. To make this precise, we need some terminology. Let $C$ be a curve over $K$ and let $P\in C$ be a point on $C$ defined over some finite extension $K(P)/K$. Then we let $h_{\cK_C}(P)$ denote the logarithmic height of $P$ with respect to the canonical divisor $\cK_C$ of $C$ and let 
\[d_K(P)=\frac{2g(K(P))-2}{[K(P):K]}\]
denote the logarithmic discriminant of $P$; here $g(K(P))$ is the
genus of $K(P)$. Then we have the following height bounds for rational
points on non-isotrivial curves due to Szpiro \cite{Szpiro} and Kim
\cite{Kim1}. 

\begin{thm}\label{SZ}
  
  Let $C$ be a non-isotrivial curve of genus at least two over a finite extension  
  $K$ of $\mathbb{F}_p(t)$. Then there are constants $B_1>0$ and $B_2$ (depending only on $C$) such that 
  \begin{equation} \label{Kim-eq}
    h_{\cK_C}(P) \leq B_1 d_K(P) + B_2
  \end{equation}
 holds for all $P\in C$.  
\end{thm}

\begin{remark} The first of these bounds (with explicit $B_1$ and $B_2$ in the semistable case) are due to Szpiro \cite[\S3]{Szpiro}, and the best possible bounds (i.e., with smallest possible $B_1$) are due to Kim \cite{Kim1}. Strictly speaking, the bound in \cite[\S3]{Szpiro} is stated for semistable curves. However, one may always pass to a finite extension $L/K$ over which $C$ is semistable \cite[\S1]{Szpiro} and thus obtain bounds of the form in \eqref{Kim-eq}. Likewise, the bound in \cite{Kim1} is stated for curves with nonzero Kodaira-Spencer class. However, the general non-isotrivial case follows from this one as follows. Assuming that $C/K$ is non-isotrivial and $\text{char}(K)=p$, there is an inseparability degree $r=p^e$ and a separable extension $L/K$ such that $C$ is defined over $L^{r}$ and that the Kodaira-Spencer class of $C$ over $L^r$ is nonzero; see \cite[pp. 51-53]{Szpiro}. Now apply Kim's theorem to $C/L^r$. In either case, Castelnuovo's inequality \cite[Theorem 3.11.3]{Stichtenoth} applied to the composite extensions $L(P)=LK(P)$ or $L^r(P)=L^rK(P)$ may be used to appropriately alter $B_1$ and $B_2$ to go from bounds with $d_L$ or $d_{L^r}$ back to those with $d_K$.             
\end{remark}

Before we apply the height bounds for points on curves from Theorem
\ref{SZ} to dynamics, we need the following elementary observation
about valuations and powers.

\begin{lem}\label{lem:valuation} 
  Let $K/\mathbb{F}_p(t)$ be finite extension and let $\ell\neq p$ be
  a prime.  Then there is a finite extension $L$ of $K$ such that if
  $u$ is any element of $K$ with the property that
  $\ell\mid v_{\fp}(u)$ for all primes $\fp$ of $K$, then $u$ is an
  $\ell$-th power in $L$.
\end{lem} 
\begin{proof} 
  Suppose that $u \in K$ is such that $\ell\mid v_{\fp}(u)$ for all
  primes $\fp$ of $K$. Then the divisor $(u)=\ell D_u$ for some
  divisor $D_u\in\text{Div}^0(K)$ of degree $0$. Hence, the linear
  equivalence class of $D_u$ is an $\ell$-torsion class in
  $\text{Cl}^0(K)$, the group of divisor classes of degree $0$. In
  particular, there are only finitely many possible linear equivalence
  classes for $D_u$ by \cite[Proposition 5.1.3]{Stichtenoth}.  Hence
  there is a finite set $\cS$ of $u \in K$ with $u = \ell D_u$ for
  some $D_u\in\text{Div}^0(K)$ such that for any $u' \in K$ with
  $u' = \ell D_{u'}$ for some $D_u\in\text{Div}^0(K)$, we have that
  $D_{u'}$ is linearly equivalent to $D_u$ for some $u \in \cS$.  Let
  $L'$ be the finite extension of $K$ generated by the $\ell$-th roots
  of the elements of $\cS$.  Now if $u$ and $u'$ are two such elements
  of $K$ as above such that $D_u$ and $D_{u'}$ are linearly
  equivalent, then $D_{u}-D_{u'}=(w_{u,u'})$ for some $w_{u,u'}\in
  K$. Hence, $u/u'=c_{u,u'}w_{u,u'}^\ell$ for some $c_{u,u'}$ in the
  field of constants of $K$. In particular, there are only finitely
  many possible such $c_{u,u'}$ since the field of constants of $K$ is
  finite.  Adjoining the $\ell$-th roots of these $c_{u,u'}$ to $L'$
  gives a finite extension $L$ of $K$.
\end{proof}   

\begin{lem}\label{from-Kim}
  Let let $\cS$ be a finite set of primes of $K$, let  $F \in \fo_{K,S}[z]$ be
  a polynomial without repeated roots and let $\ell\neq p$ be a prime
  such that $C: y^\ell = F(x)$ is a non-isotrivial curve of genus
  $g(C) > 1$.  Then there are constants $r_1>0$ and $r_2$ (depending
  on $F$, $\ell$, $K$, and $\cS$) such that
  \begin{equation}\label{lower}
    \sum_{\substack{v_{\fp}(F(a)) > 0 \\ \ell \nmid v_{\fp}(F(a))}} N_\fp
     \geq r_1 h(a)+r_2
   \end{equation}
holds for all $a \in \fo_{K,S}$.  
\end{lem}
\begin{proof} Suppose that $C: y^\ell = F(x)$ is a non-isotrivial curve of genus $g(C) > 1$. Then given $a \in \fo_{K,S}$, we let $u_a:=F(a)$ and choose  a corresponding point $P_a=\big(a,\sqrt[\ell]{u_a}\big)$ on $C$. From here, we proceed in cases. 

Suppose first that $\ell\mid v_{\fp}(u_a)$ for all primes $\fp$ of $K$. Then by Lemma \ref{lem:valuation} there exists a finite extension $L/K$ (independent of $a$) such that $u_a$ is an $\ell$th power in $L$. In particular, since we may assume that $L$ contains a primitive $\ell$th root of unity, $K(P_a)\subseteq L$. Therefore, \eqref{Kim-eq} implies that $h_{\cK_C}(P_a)$ is absolutely  bounded. However, the canonical divisor class is ample in genus at least $2$, so that the set of possible points $P_a$ is finite in this case. Therefore, $h(a)$ is bounded and \eqref{lower} holds trivially (take $r_1=1$ and choose $r_2$ to be sufficiently negative).

Now suppose that there exists a prime $\fp$ of $K$ such that $\ell\nmid v_{\fp}(u_a)$. Then we may apply the genus formula in \cite[Corollary 3.7.4]{Stichtenoth} to deduce that                              
\begin{equation}\label{disc1}
\begin{split} 
d(P_a)&=2g(K)-2+\frac{1}{\ell}\sum_{\fp}\big(\ell-\gcd(\ell,v_{\fp}(u_a)\big)N_{\fp}\\[8pt] 
&=2g(K)-2+\Big(\frac{\ell-1}{\ell}\Big)\sum_{\substack{v_{\fp}(u_a) > 0 \\ \ell \nmid v_{\fp}(u_a)}}N_\fp \;\; 
+\;\; \Big(\frac{\ell-1}{\ell}\Big)\sum_{\substack{v_{\fp}(u_a)<0 \\ \ell \nmid v_{\fp}(u_a)}}N_\fp \\[8pt]
&\leq 2g(K)-2+ \Big(\frac{\ell-1}{\ell}\Big)\sum_{\substack{v_{\fp}(u_a) > 0 \\ \ell \nmid v_{\fp}(u_a)}}N_\fp \;\; 
+ \Big(\frac{\ell-1}{\ell}\Big)\sum_{\fp\in \mathcal{S}}N_{\fp},    
\end{split}  
\end{equation} 
since the only way that $u_a:=F(a)$ can have negative valuation at $\fp$ is if $\fp\in \mathcal{S}$. However, this is a finite set of primes. Therefore, \eqref{disc1} implies that 
\begin{equation}\label{disc2}
d(P_a)\leq\Big(\frac{\ell-1}{\ell}\Big)\sum_{\substack{v_{\fp}(F(a)) >0 \\ \ell \nmid v_{\fp}(F(a))}}N_\fp + O_{K,F,\mathcal{S}}(1). 
\end{equation}
On the other hand, if $\pi:C\rightarrow\mathbb{P}^1$ is the map given by projection onto the $x$-coordinate, then $\pi$ pulls back a degree one divisor on $\mathbb{P}^1$ (yielding the Weil height on $\mathbb{P}^1$) to a degree $\ell$ divisor on $C$. Hence, the algebraic equivalence of divisors and \cite[Thm III.10.2]{SilvAdvanced} together imply that   
\[
\lim_{h_{\cK_C}(P)\rightarrow\infty} \frac{h(\pi(P))}{h_{\cK_C}(P)}=\frac{\ell}{2g(C)-2}.
\]
In particular, we may deduce that  
\begin{equation}\label{height-div}
h(a)\leq \frac{(1+\epsilon)\ell}{(2g(C)-2)}h_{\cK_C}(P_a)+O_{K,F,\ell,\epsilon}(1)
\end{equation} 
for all $\epsilon>0$ and all $a\in K$ (not just $a\in \fo_{K,S}$). Finally, by choosing $\epsilon=1$ and combining \eqref{Kim-eq}, \eqref{disc2}, and\eqref{height-div}, we see that 
there are constants $r_1>0$ and $r_2$ (depending on $F$, $\ell$, $K$, and $\cS$) such that
\[ \sum_{\substack{v_{\fp}(F(a)) > 0 \\ \ell \nmid v_{\fp}(F(a))}} N_\fp
     \geq r_1 h(a)+r_2
\]
holds for all $a \in \fo_{K,S}$. In particular, after replacing $r_1$ and $r_2$ with the minimum of the corresponding constants from the first and second cases above, we prove Lemma \ref{from-Kim}.         
\end{proof}

\begin{lem}\label{delta}
  Let $f \in K[z]$ be a non-isotrivial polynomial with $\deg f = d > 1$ and
  let $\alpha, \gamma \in K$ where $\gamma$ is not postcritical and
  $\alpha$ is not preperiodic.
  Then for any prime $\ell\ne p$, there is a $\delta >
  0$ such that for all sufficiently
  large $n$, we have
  \begin{equation}
  \sum_{\substack{v_{\fp}(f^n(\alpha) - \gamma) > 0 \\ \ell \nmid
      v_{\fp}(f^n(\alpha) - \gamma)}} N_\fp \geq \delta  d^n h_f(\alpha).
\end{equation}
\end{lem}
\begin{proof}
  Let $S$ be finite set of primes such that $\alpha$, $\gamma$, and all the
  coefficients of $f$ are in $\fo_{K,S}$.  Then $f^n(\alpha) \in
\fo_{K,S}$ for all $m$.  By Theorem \ref{non-curve}, there is an $m$ such that
the curve given by 
\[ y^\ell = \prod_{\substack{ \gamma \in \Kbar \\ f^m(\gamma) =
      \beta}}(x- \beta) \]
is not an isotrivial curve.    There is an $\omega \in K$ (the
leading term of $f^m(z) - \gamma$) and an $e$ (coming from the degree
of inseparability of $f^\ell$) such that
\[  f^m(z) - \gamma = \omega \prod_{\substack{ \gamma \in \Kbar \\ f^m(\gamma) =
      \beta}}(z- \beta)^{p^e}. \]
Let 
\[ F(z) = \prod_{\substack{ \gamma \in \Kbar \\ f^m(\gamma) =
      \beta}}(z- \beta). \]

Applying Lemma \ref{from-Kim} with $a = f^{n-m}(\alpha)$ we see that
since $\ell \not= p$, we have constants $r_1, r_2$ such that
\begin{equation*}
  \sum_{\substack{v_{\fp}(f^n(\alpha) - \gamma) > 0 \\ \ell \nmid
      v_{\fp}(f^n(\alpha) - \gamma)}} N_\fp \geq
 \left( \sum_{\substack{v_{\fp}(F(a)) > 0 \\ \ell \nmid
      v_{\fp}(F(a))}} N_\fp \right) - h(\omega)
     \geq r_1 h(f^{n-m}(\alpha))+r_2  - h(\omega).
 \end{equation*}
Since $|h_f - h| \leq O(1)$ and $h_f(f^{n-m}(\alpha)) = d^{n-m}
h_f(\alpha)$, we see that there is a constant $r_3$ such that
\[ \sum_{\substack{v_{\fp}(f^n(\alpha) - \gamma) > 0 \\ \ell \nmid
      v_{\fp}(f^n(\alpha) - \gamma)}} N_\fp \geq r_1 d^{n-m}
  h_f(\alpha) + r_3 \]
for all $n$.  Choosing a $\delta$ such that $0 < \delta < r_1/d^m$
then gives
\[ \sum_{\substack{v_{\fp}(f^n(\alpha) - \gamma) > 0 \\ \ell \nmid
      v_{\fp}(f^n(\alpha) - \gamma)}} N_\fp \geq \delta d^n
  h_f(\alpha) \]
for all sufficiently large $n$, as desired.
  \end{proof}

We are now ready to prove Proposition \ref{most-gen}.

\begin{proof}[Proof of Proposition \ref{most-gen}]

We first note it suffices to prove this after passing to a finite
extension of $K$ since $\ell \not= p$.  To see this, let $L$ be a finite extension of $K$,
let $L^s$ denote the separable closure of $K$ in $L$, and let $\fq$ be
a prime in $L$ lying over a prime $\fp$ of $K$. Then
$v_\fq(f^n(\alpha) - \beta) =  [L:L^s] v_\fp(f^n(\alpha) -
\beta)$ unless $\fp$ is in the finite set of primes of $K$ that ramify
in $L^s$.  We also note that $h_f(\alpha) > 0$ since $\alpha$ is not
preperiodic and $f$ is not isotrivial, by \cite[Corollary
    1.8]{Baker1}.    

We change coordinates so that $\beta = 0$.  Let $r$ be the smallest
positive integer such that $f^r(\gamma) = 0$.
After passing to a finite extension we may assume that all the roots
of $f^r(z)$ are in $K$.  Let $e = e_{f^r}(\gamma/\beta)$
 and write
\begin{equation*}
f^r(z) = (z-\gamma)^e g(z).
\end{equation*}
 Then for all but finitely many primes $\fp$ of $K$ we have 
\begin{equation}\label{e}
 v_\fp(f^{n+r}(\alpha)) = e v_\fp(f^{n}(\alpha) - \gamma) 
\end{equation}
for all $n$.

Since $\gamma$ is not post-critical, by Lemma \ref{delta}, there exists $\delta >
  0$ such that for all sufficiently
  large $n$, we have
  \begin{equation}\label{d}
  \sum_{\substack{v_{\fp}(f^n(\alpha) - \gamma) > 0 \\ \ell \nmid
      v_{\fp}(f^n(\alpha) - \gamma)}} N_\fp \geq \delta  d^n h_f(\alpha).
\end{equation}

Let $\cW$ be the roots of $f^r(z)$ that are not roots of $f^{r'}(z)$
for any $r' < r$.  Let $\cS_1$ be the set of primes of bad reduction
for $f$ and let $\cS_2$ be the set of primes such that $v_\fp(f^{r'}(w))
> 0$ for some $r' < r$ and some $w \in \cW \cup \{ \alpha \}$.   Now, for each $n$, let $\cY(n)$ be set of
primes $\fp$ such that $v_\fp(f^n(\alpha) - \gamma) > 0$ and
$v_\fp(f^{n'}(\alpha)) > 0$ for some $n' < n + r$.   If $\fp \notin
\cS_1 \cup \cS_2$, then
$v_\fp(f^m(\alpha)) - \gamma') > 0$ for some $\gamma' \in \cW$ and some $m < n$.  Thus,
since $\gamma$ is not in the forward orbit of any
element of $\cW$ and the sets $\cW$, $\cS_1$, and $\cS_2$ are all finite, we may apply Lemma
\ref{old} to each element of $\cW$.  We obtain
\begin{equation}\label{Y}
  \sum_{\fp \in \cY(n)} N_\fp \leq \frac{\delta}{2} d^n h_f(\alpha)
\end{equation}
for all sufficiently large $n$.   Combining \eqref{Y} with  \eqref{e}
and \eqref{d}, we see that for all sufficiently large $n$, there is a
prime $\fp$ such that
\begin{itemize}
\item $v_{\fp}(f^n(\alpha) - \gamma) > 0$;
\item $\ell \nmid
      v_{\fp}(f^n(\alpha) - \gamma)$; 
    \item $v_\fp(f^{n'}(\alpha)) = 0$ for all $0 < n' < n$; and
      \item $v_\fp(f^{n+r}(\alpha)) = e v_\fp(f^{n}(\alpha) - \gamma)$.
 \end{itemize}
Since $e$ is prime to $\ell$, it follows that the Zsigmondy set $\cZ(f, \alpha, \beta, \ell )$ is
   finite. 
\end{proof} 

\section{Applications}\label{applications}

The original Zsigmondy theorem \cite{Bang,Zsigmondy} had to do with
orders of algebraic numbers modulo primes.  We can treat a related
dynamical problem; here we will not assume non-isotriviality.  We begin with some notation and terminology.  If $\alpha \in K$ is an integer at a
prime $\fp$, we let $\alpha_\fp \in k_\fp$ be its reduction at $\fp$.
If $f \in K[x]$, and all of the coefficients of $f$ are integers at
$\fp$, we let $f_\fp \in k_\fp$ be the reduction of $f$ at $\fp$
obtained by reducing each coefficient of $f$ at $\fp$.  if
$g: \cU \lra \cU$ is any map from a set to itself and $u \in \cU$ is
periodic under $g$, then the {\bf prime period} of $u$ for $g$ the
smallest positive integer $m$ such that $g^m(u) = u$.  We say that a
polynomial in $f \in K[x]$ is additive if
$f(\alpha + \beta) = f(\alpha) + f(\beta)$ for all
$\alpha, \beta \in \Kbar$.

\begin{thm}
  Let $f$ be a polynomial of degree greater than 1 and let
  $\alpha \in K$ be a point that is not preperiodic for $f$.  If $f$
  is not both isotrivial and additive, then for all but finitely many
  positive integers $n$, there is a prime $\fp$ such that the prime period
  of $\alpha_\fp$ for $f_\fp$ is equal to $n$.  If $f$ is isotrivial
  and additive, then there for all but finitely many positive integers
  $n$ that are not a power of $p$, there is a $\fp$ such that
  the prime period of $\alpha_\fp$ for $f_\fp$ is equal to $n$
\end{thm}
\begin{proof}
  If $f$ is not isotrivial, this follows immediately from Theorem
  \ref{Z2} by letting $\alpha = \beta$.  If $f$ is
  isotrivial, then after a change of coordinates, we may assume that
  $f \in k[x]$ and $\alpha \in K \setminus k$ for some finite
  extension $k$ of $\bF_p$.  If $f$ is not additive
  then for all but finitely many positive integers $n$, there exists
  $\beta_n \in k$ having prime period $n$ for $f$, by
  \cite[Theorem]{Pezda}.  For each such $\beta_n$, there exists $\fp_n$
  such that $\alpha_{\fp_n} = \beta_n$, so we see that for all but all
  but finitely many positive integers $n$, there exists $\fp$ such that
  $\alpha_\fp$ for $f_\fp$ is equal to $n$.  If $f$ is additive, then
  for all but finitely many positive integers $n$ that are not a power
  of $p$, there exists $\beta_n \in k$ having prime period $n$ for $f$,
  by \cite[Theorem]{Pezda}.  Then, as in the non-additive case, we may
  choose $\fp_n$ such that $\alpha_{\fp_n} = \beta_n$.
\end{proof}

Theorem \ref{main} allows one to prove characteristic $p$ analogs of
various results that rely on the results of \cite{SilInt}.  For
example, the proofs of Theorems 4 and 5 of \cite{Scanlon} extend
easily to the case of non-isotrivial rational functions over a
function field in characteristic $p$, using Theorem \ref{main}.
Similarly, one can use Theorem \ref{main} to prove Theorem 4 of
\cite{BIJJ} with the additional hypothesis that at least one of the wandering
critical points of $\varphi$ has a ramification degree that is not a
power of $p$.

We will now prove a few results that about unicritical polynomials
that rely on Theorem \ref{Z1}, which is not available over number
fields.

The following lemma is very similar to \cite[Proposition 3.1]{BT2}; we
include the proof for a sake of completeness.   

\begin{lem}\label{U}
  Let $f(x) = x^d + c$ where $d$ is an integer greater than than 1
  that is not divisible by $p$, let $\beta \in K$, and let $n$ be a
  positive integer.  Let $\fp$ be any prime of $K$ such that
  \begin{itemize}
  \item[(i)] $|c|_\fp \leq 1$;
  \item[(ii)] $|\beta|_\fp \leq 1$; and
  \item[(iii)] $|f^m(0)|_\fp = 1$ for all $0 \leq m \leq n$.
  \end{itemize}
Then $\fp$ does not ramify in $K(f^{-n}(\beta))$. 
  \end{lem}
  \begin{proof}
    We proceed by induction.  The case where $n=1$ follows immediately
    from taking the discriminant of $x^d + (c - \beta)$.  Now, let
    $\fp$ be a prime satisfying (i) -- (iii) for some $n \geq 2$.
    Then it also satisfies them for $n-1$, so by the inductive
    hypothesis, the prime $\fp$ does not ramify in
    $K(f^{-(n-1)}(\beta))$.  Now, $K(f^{-n}(\beta))$ is obtained from
    $K(f^{-(n-1)}(\beta))$ by adjoining elements of the form $\sqrt[d]{c -
      \gamma_i}$ for $f^{n-1}(\gamma_i) = \beta$.   For any prime
    $\fq$ in $K(f^{-(n-1)}(\beta))$ lying over $\fp$, we see that
    $|\gamma_i|_\fq \leq 1$ by (i) and (ii) and $|\gamma_i| \geq 1$ by
    (i), (ii), and (iii).  Thus, each $\fq$ in $K(f^{-(n-1)}(\beta))$
    lying over $\fp$ does not ramify in any $K(f^{-(n-1)}(\beta)) (\sqrt[d]{c -
      \gamma_i})  = K(f^{-n}(\beta))$.
    Since each such $\fq$ does not ramify over $\fp$ by the inductive
    hypothesis, it follows that $\fp$ does not ramify in
    $K(f^{-n}(\beta))$, as desired. 
    
\end{proof}

The next lemma follows a proof that is similar to that of
\cite[Proposition 3.2]{BT2} and \cite[Theorem 5]{BIJJ}.  

\begin{lem}\label{R}
  Let $f(x) = x^d + c$ where $c \in K \setminus k$ where $d$ is an
  integer greater than than 1 that is not divisible by $p$.  Let
  $\beta \in K$, let
  $\ell \not= p$ be a prime number, and let $e$ be a positive integer
  such that $\ell^e$ divides $d$.  Suppose that $\fp$ is a primitive
  $\ell$-divisor of $f^n(0) - \beta$ such that
  $|c|_\fp = |\beta|_\fp = 1$.  Then for any prime $\fp'$ in
  $K(f^{-(n-1)}(\beta))$ that lies over $\fp$, there is a prime $\fq$
  in $K(f^{-n}(\beta))$ such that $\ell^e$ divides $e(\fq/\fp')$.
\end{lem}
\begin{proof}
  Let $\fp'$ be a prime in $K(f^{-(n-1)}(\beta))$ lying over $\fp$.
  By Lemma \ref{U}, the prime $\fp$ does not ramify in
  $K(f^{-(n-1)}(\beta))$, so $v_{\fp'}(z) = v_\fp(z)$ for all $z \in
  K$. Since $f^n(0) - \beta$ = $\prod_{f^{n-1}(\gamma) = \beta} f(0) -
  \gamma$, we see that there is some $\gamma \in f^{-(n-1)}(\beta)$
  such that $\ell \nmid v_{\fp'}(c - \gamma)$.  Thus, if $\fq$ is a prime of
  $K(f^{-(n-1)}(\beta))(\sqrt[d]{c - \gamma})$ lying over $\fp'$, we
  see that $\ell^e | e(\fq/\fp')$.  
\end{proof}

Using the Lemmas above, we can prove a result for separable non-isotrivial
polynomials of the form $x^d + c$ that is a special case of a
characteristic $p$ analog of \cite[Theorem 1.1]{BT2}.  Note that if
$f(x) = x^d + c$ and $d$ is not divisible by $p$, then $f$ is isotrivial
if and only if $c \in \Fpbar$.  To see this, note that
$h_f(0) = \frac{h(c)}{d} > 0$ when $c \notin \Fpbar$, as can be seen
by simply considering the orbit of $f$ at the places $v$ where $|c|_v > 1$.
Therefore, if $c \notin \Fpbar$, then $f$ has a critical point that is
not preperiodic, and hence $f$ cannot be isotrivial.  We note also
that a polynomial of the form $x^d + c$ is separable if and only if $p
\nmid d$.  

 \begin{thm}\label{Z-ram}
   Let $f(x) = x^d + c$ be a separable non-isotrivial polynomial of
   degree $d  > 1$.  Let $\beta \in K$.
   Then for all sufficiently large $n$, there is a prime $\fp$ of $K$
   such that $\fp$ ramifies in $K(f^{-n}(\beta))$ but not in
   $K(f^{-(n-1)}(\beta))$.
\end{thm}
\begin{proof}
  Note that $\beta$ cannot be exceptional since $c \not= 0$. 
  Let $\ell \not= p$ be a prime dividing $d$.  By Theorem \ref{Z1},
  for all sufficiently large $n$, there is a prime $\fp$ such that
  $v_\fp (f^n(0) - \beta) > 0$ with
  $\ell \nmid v_\fp (f^n(0) - \beta)$ and $v_\fp (f^m(0) - \beta) = 0$
  for all $0 < m < n$.  Since $|c|_\fp = |\beta|_\fp = 1$ for all but
  finitely many $\fp$ we may also suppose that $|c|_\fp = |\beta|_\fp
  = 1$.  Then, by Lemma \ref{U}, the prime $\fp$ does not ramify in
  $K(f^{-(n-1)}(\beta))$.  By Lemma \ref{R}, it does ramify in
  $K(f^{-n}(\beta))$.  
\end{proof}

The next result is a characteristic $p$ analog of a theorem of Pagano
\cite[Theorem 1.3]{Pagano} for number fields (see also \cite{Growth}
for a similar result); the growth condition here is stronger than what
Pagano obtains over number fields.
 
 \begin{thm}\label{growth}
   Let $f(x) = x^d + c$ be a separable non-isotrivial polynomial of
   degree $d > 1$.  Let $\beta \in K$.  Then there is a constant
   $C(n, \beta) > 0$ such that $[K(f^{-n}(\beta)): K] > C(n, \beta)d^n$ for all
   positive integers $n$.
 \end{thm}
 \begin{proof}
   It will suffice to show that
   $d$ divides
   $[K(f^{-n}(\beta)): K(f^{-(n-1)}(\beta))]$ for all sufficiently
   large $n$.  Let $\ell$ be a prime such that $\ell^e | d$ for some
   $e > 0$.  Applying Theorem \ref{Z1} as in Theorem \ref{Z-ram}, we
   see that for all sufficiently large $n$, there is a prime $\fp$
   with the property $|c|_\fp = |\beta|_\fp = 1$ such that
   $v_\fp (f^n(0) - \beta) > 0$ with
   $\ell \nmid v_\fp (f^n(0) - \beta)$ and
   $v_\fp (f^m(0) - \beta) = 0$ for all $0 < m < n$.  The Lemma
   \ref{R} implies that for any prime $\fp'$ in $K(f^{-(n-1)}(\beta))$
   that lies over $\fp$, there is a prime $\fq$ in $K(f^{-n}(\beta))$
   such that $\ell^e$ divides $e(\fq/\fp')$.  Hence
   $\ell^e | [K(f^{-n}(\beta)): K(f^{-(n-1)}(\beta))]$.  Since this
   holds for any prime $\ell$ such that $\ell^e | d$ for some $e > 0$,
   it follows that $d | [K(f^{-n}(\beta)): K(f^{-(n-1)}(\beta))]$ for
   all sufficiently large $n$, and our proof is complete.
  \end{proof}
 
  We are can now prove a finite index result for iterated monodromy
  groups of quadratic polynomials.  We need a little terminology to
  state our result.    

  Let $L$ be a field.  Let $f$ be a quadratic polynomial and let
  $\beta\in \overline L$.  For $n\in\N$, let $L_n(f,\beta)=L(f^{-n}(\beta))$ be
  the field obtained by adjoining the $n$th preimages of $\beta$ under
  $f$ to $L(\beta)$.  and let
  $L_\infty(f,\beta)=\bigcup_{n=1}^\infty L_n(f,\beta)$.  We let
  $G_\infty(\beta)=\Gal(L_\infty(f,\beta)/L)$.  The group
  $G_\infty(\beta)$ embeds into $\Aut(T^2_\infty)$, the automorphism
  group of an infinite $2$-ary rooted tree $T^2_\infty$ (note that all
  of the definitions here generalize to rational functions of any degree 
  -- see \cite{OdoniIterates} or \cite{JKMT}, for example).  Boston and Jones
  \cite{BostonJonesArboreal} asked if $G_\infty(\beta)$ had finite
  index in $\Aut(T^2_\infty)$ whenever $f$ is not post-critically
  finite in the case where $L$ is a number field.  It was later shown
  \cite{JuulEtAl} that this is true if the pair $(f,\beta)$ is
  eventually stable (see below), assuming the $abc$ conjecture.  This
  was also shown to be true unconditionally for non-isotrivial
  quadratic polynomials over function fields of characteristic 0 in
  \cite{QuadFin}.

For $\beta\in {\overline L}$ and a polynomial $f \in L[x]$, the pair $(f,\beta)$ is said to be
{\bf eventually stable} if the number of irreducible factors of
$f^n(x)-\beta$ over $L(\beta)$ is bounded independently of $n$ as
$n\to\infty$ (stability and eventual stability can also be defined for
rational functions as in~\cite{RafeAlonEventualStability}).  We will
prove a finite index result for non-isotrivial quadratic polynomials
over function fields of odd positive characteristic under an eventual
stability assumption.

  The technique we use is the same as that
  used in \cite{QuadFin} (see also \cite{JuulEtAl, BT, HJ}).  We make use
  of \cite[Proposition 7.7]{QuadFin}, which is stated in
  characteristic 0 but is true with no changes in the proof in
  characteristic $p$ provided that $K(f^{-n}(\beta))$ is separable over
  $K$ for all $n$, which is automatic here when $p > 2$; the following result is a strengthening of \cite[Corollary 1]{London}. 

\begin{thm}\label{fin-index}
  Let $f$ be a non-isotrivial quadratic polynomial defined over a field
  $K$ that is a finite extension of $\Fp(t)$.  Suppose that $p>2$
  and that $\beta$ is not post-critical or periodic for $f$.  Suppose
  furthermore that the pair $(f,\beta)$ is eventually stable.  Then
  $G_\infty(\beta)$ has finite index in $\Aut(T^2_\infty)$.  
  \end{thm}
  \begin{proof}
    As in \cite{QuadFin}, it will suffice to show that for all sufficiently large $N$, we
    have
    \[ Gal(K_N/K_{N-1}) \cong C_2^{2^N},\]
    where $C_2$ is the
    cyclic group with two elements.  After a change of
    variables, we may assume that $f(x) = x^2 + c$ for some
    $c \in K \setminus k$.   
    
    Since $(f,\beta)$ is eventually stable, there is an $m$ such that
    $f^m(x) - \beta = (x-\gamma_1) \cdots (x-\gamma_{2^m})$ for
    $\gamma_i$ with the property that $f^n(x) - \gamma_i$ is
    irreducible over $K(\gamma_i)$ for all $n$ for $i=1, \dots, 2^m$, by
    \cite[Proposition]{BT}.  Let
    $L = K(\gamma_1, \dots, \gamma_{2^m})$.  It follows from
    \cite[Proposition 7.7]{QuadFin} and Lemma \ref{R} that we must
    have $\Gal(K_{n+m}/K_{n+m-1}) \cong [C_2]^{2^{m+n}}$ whenever
    there are primes $\fp_i$ of $L$, for $i=1, \dots, 2^m$, such that
   \begin{itemize}
   \item[(i)] $v_{\fp_i}(c) = v_{\fp_i}(\gamma_j) = 0$ for $j = 1,
     \dots 2^m$;
   \item[(ii)] $2 \nmid v_{\fp_i}(f^n(0) - \gamma_i)$;
     \item[(iii)] $v_{\fp_i}(f^{n'}(0) - \gamma_i) = 0$ for all $n'
       < n$; and
     \item[(iv)]  $v_{\fp_i}(f^{n'}(0) - \gamma_j) = 0$ for all $n'
       \leq n$ and $j \not= i$; 
     \end{itemize}

Note that condition (i) holds for all but finitely many primes
$\fp_i$.  Hence, we will be done if we can show that for all
sufficiently large $n$, there are $\fp_i$, for $i=1, \dots, 2^m$, that
satisfy conditions (ii), (iii), and (iv).  
     
   Now, fix a $\gamma_i$.  By Lemma \ref{delta},  there exists $\delta >
  0$ such that for all sufficiently
  large $n$, we have
  \begin{equation}\label{fin1}
  \sum_{\substack{v_{\fp}(f^n(0) - \gamma_i) > 0 \\ 2 \nmid
      v_{\fp}(f^n(0) - \gamma_i)}} N_\fp \geq \delta  d^n h_f(0).
\end{equation}
For any $n$, let $\cX(n)$ be the set of primes $\fp$ such that
$v_\fp(f^n(0) - \gamma_i) > 0$ and $v_\fp(f^{n'}(0) - \gamma_i) > 0$
for some $n' < n$.  Since $\gamma_i$ is not periodic and $h_f(\alpha)
> 0$, we may apply
Lemma \ref{old}.  We see then that for all sufficiently large $n$, we
have
   \begin{equation}\label{fin2} \sum_{\fp \in \cX(n)} \leq
     \frac{\delta}{3} d^n h_f(0).
   \end{equation}
For any $n$ and $i \not =j$, we let $\cY_j(n)$ be the set of primes
   $v_\fp(f^n(0) - \gamma_j) > 0$ and
   $v_\fp(f^{n'}(0) - \gamma_j) > 0$ for some $n' \leq n$.
   Since
   $f^{n'}(\gamma_j) \not= \gamma_i$ for all $n'$ and $i \not= j$, we may
   apply Lemma \ref{delta} again.  Since in addition we have
   $v_\fp(\gamma_i - \gamma_j) \not= 0$ for all but finitely many
   $\fp$ when $i \not= j$, we see that for all sufficiently large $n$,
   we have
   \begin{equation}\label{fin3} \sum_{j \not= i} \sum_{\fp \in \cY_j(n)} \leq
     \frac{\delta}{3} d^n h_f(0).
   \end{equation}
   Since $\delta h_f(0) > 0$, equations \eqref{fin1}, \eqref{fin2},
   and \eqref{fin3} imply that for any sufficiently large $n$, there
   is a prime $\fp_i$ satisfying conditions (ii), (iii), and (iv), and our proof
   is complete.
 \end{proof}
 
 \begin{remark}
   We note that while conditions (i) and
 (ii) above are weaker as stated than Condition R from \cite[Definition
 7.2]{QuadFin}, they do imply that the prime $\fp_i$ ramifies in
 $K(f^{-n}(\gamma_i))$ (by Lemma \ref{R}), which is what \cite[Proposition
 7.7]{QuadFin} requires.
\end{remark}

It should also be possible to prove a finite index result along the
lines of Theorem \ref{fin-index} more generally for non-isotrivial
polynomials of the form $x^d + c$, where $d > 2$ and $p \nmid d$ by
modifying techniques in \cite{QuadFin} and combining them
with our argument for Theorem \ref{growth} above.

\bibliographystyle{amsalpha}
\bibliography{IntBib}

\end{document}